\documentclass[reqno]{amsart}

\usepackage{amsmath,amsthm,amssymb,amscd,mathtools}
\usepackage{tabularx}
\usepackage{enumitem}
\usepackage{comment}
\usepackage{url}
\usepackage{hyperref}
\usepackage{parskip}
\newcommand{\N}{\mathbb{N}}
\newcommand{\Z}{\mathbb{Z}}
\newcommand{\Q}{\mathbb{Q}}

\newcommand{\F}{\mathbb{F}}

\renewcommand{\P}{\mathbb{P}}
\newcommand{\p}{\mathfrak{p}}

\DeclareMathOperator{\Aut}{Aut}

\newcommand{\Gal}{\operatorname{Gal}}
\newcommand{\Frac}{\operatorname{Frac}}
\renewcommand{\o}{\text{o}}
\renewcommand{\char}{\text{char}}
\usepackage{colonequals}
\DeclareMathOperator{\Spec}{Spec}
 \newtheorem{thm}{Theorem}[section]
 \newtheorem{cor}[thm]{Corollary}
 \newtheorem{lem}[thm]{Lemma}
 \newtheorem{conj}[thm]{Conjecture}
 \newtheorem{prop}[thm]{Proposition}

 \theoremstyle{definition}
 \newtheorem{defn}[thm]{Definition}
 \theoremstyle{remark}
 \newtheorem{rem}[thm]{Remark}
 \newtheorem*{ex}{Example}
 \numberwithin{equation}{section}

\newcommand{\embedding}{\lhook\joinrel\relbar\joinrel\relbar\joinrel\rightarrow}

\begin{document}
\title[On $p$-adic denseness of quotients of values of integral forms]
{On $p$-adic denseness of quotients of values of integral forms}

\author[Deepa Antony]{Deepa Antony}

\address{
	Department of Mathematics \\
	Indian Institute of Technology Guwahati \\
	Assam, India, PIN- 781039}

\email{deepa172123009@iitg.ac.in}

\author[Rupam Barman]{Rupam Barman}
\address{Department of Mathematics \\
	Indian Institute of Technology Guwahati \\
	Assam, India, PIN- 781039}
\email{rupam@iitg.ac.in}

\author[Stevan Gajovi\'{c}]{Stevan Gajovi\'{c}}
\address{Universidad Politécnica de Madrid, Campus Montegancedo, 28660, Boadilla del Monte, Madrid, Spain}
\address{Max Planck Institute for Mathematics, Vivatsgasse 7, 53111 Bonn, Germany}
\address{Faculty of Mathematics and Physics, Charles University, Department of Algebra
Sokolovská 49/83, 186 25 Praha 8, Czech Republic}
\email{stevangajovic@gmail.com}

\author[Daniel \v{S}irola]{Daniel \v{S}irola}
\address{Mathematical Institute of the University of Bonn, Endenicher Allee 60, 53115 Bonn, Germany}
\email{daniel.sirola.math@gmail.com}




\begin{abstract}
Given $A\subseteq \Z$, the ratio set or the quotient set of $A$ is defined by $R(A):=\{a/b: a, b\in A, b\neq 0\}$. It is an open problem to study the denseness of $R(A)$ in the $p$-adic numbers when $A$ is the set of values attained by an integral form. For a given form, we investigate whether this happens for all but finitely many $p$. We consider the more general question when forms have coefficients in Dedekind domains, and, under certain conditions, we prove that the analogous statement holds. However, we also give examples of integral forms for which the answer is negative, and it is crucial that the degree of such forms is composite. We conjecture that we cannot find such examples with forms of prime degree having sufficiently many variables, which is indeed the case when the degree is 2, 3, or 5.  Our innovation is to consider the problem in terms of varieties defined by forms, thereby using the tools from algebraic geometry, making their first appearance in this setting. We construct integral forms such that the ratio set of its values is dense in at least one $\Q_p$ but only in finitely many of them. 
\end{abstract}

\maketitle
\section{Introduction and statement of results} 

\subsection{General introduction and known results}
Given $A\subseteq \Z$, the set $R(A)=\{a/b:a,b\in A, b\neq 0\}$ is called the \emph{ratio set} or quotient set of $A$. A well-known problem is to characterise all subsets of $\mathbb{N}$ whose ratio sets are dense in the positive real numbers. Many authors have studied this problem, see for example \cite{real-1, real-2, real-3, real-4, real-5, real-6, real-7, real-8, real-9, real-10, real-11, real-12, real-13, real-14}.
\par For a prime $p$, let $\mathbb{Q}_p$  denote the field of $p$-adic numbers. In recent years, the denseness of ratio sets in $\mathbb{Q}_p$ has been studied by several authors, 
see for example \cite{cubic, reccurence, diagonal, Donnay, garciaetal, garcia-luca, miska, piotr, miska-sanna, Sanna}. Several interesting subsets of the set of integers were considered, for example, the set of integers formed by linear recurrence relations, subsets of $\mathbb{N}$ which partition  $\mathbb{N}$, sequences in geometric progression, images of integers under polynomials, etc. 

In this article, we study the denseness of the ratio set of values attained by homogeneous polynomials in several variables, also known as forms.

\par A \emph{form} $F$ of degree $n$ with coefficients in a ring $T$ is a homogeneous polynomial
\begin{align*}
	F(x_1,x_2, \ldots, x_k)=\sum_{1\leq i_1\leq i_2\leq \ldots\leq i_n\leq k}a_{i_1i_2\ldots i_n}x_{i_1} x_{i_2} \cdots x_{i_n}, 
\end{align*} 
where all $a_{i_1i_2\ldots i_n}\in T$. When $T=\Z$, we call forms \emph{integral}.  The ratio set generated by an integral form $F$ over $T$ is 
\begin{align*}
	R(F(T))=\{F(\overline{x})/F(\overline{y}): \overline{x}, \overline{y}\in T^k, F(\overline{y})\neq 0\}.
\end{align*}

Given an integral form $F$, studying the denseness of $R(F(T))$ in $\Q_p$ seems to be a difficult problem; it is still an open problem. This problem has been completely answered by Donnay, Garcia and Rouse \cite{Donnay} for quadratic forms and also by Miska \cite{miska} with a shorter proof. They showed that for a non-singular quadratic form $Q$, $R(Q(\Z))$ is dense in $\mathbb{Q}_p$ if and only if $Q$ has at least three variables, or, if $Q$ is a binary form, the discriminant of $Q$ is a non-zero square in $\mathbb{Q}_p$. 
\par In a  recent article \cite{cubic}, the first two authors studied the denseness of $R(C(\Z))$ in the $p$-adic numbers when $C$ is the diagonal cubic form $C(x,y)=ax^3+by^3$, where $a$ and $b$ are non-zero integers. Jointly with Miska \cite{diagonal}, the first two authors extended the results on the diagonal binary cubic forms to the diagonal binary forms $ax^n+by^n$ for all $n\geq 3$. It was also shown that if a form $F$ is anisotropic modulo $p$, then $R(F(\Z))$ is not dense in $\mathbb{Q}_p$.  Finally, in the article \cite{piotr}, Miska, Murru and Sanna studied the denseness of $R(f(\Z))$, where $f$ is a polynomial, and they give sufficient conditions for the positive answer, see Theorems \ref{thm:polynomial-coprime-multiplicities-roots} and Theorem \ref{thm6}. 

\subsection{The content and main results}\label{subsec:our-results} 
The novelty in our approach is that we also consider varieties defined by forms $F$, which allows us to apply tools from algebraic geometry, explained in Section~\ref{sec:ag-cones}, which, to our knowledge, have not previously been used in this context. Section~\ref{sec:ag-cones} contains the relevant statements, which might be known, but, Theorem~\ref{thm:cone-over-Fp-p-large-over-Q} is, we believe, a new result, and also used in the proofs of our results.

We first state several useful necessary conditions for the denseness, or not, of the ratio set of a form $F$. Since several similar results appeared before, we generalise them over general nonarchimedean complete discretely valued fields, with the restriction to the finiteness of the residue field when necessary, see \S\ref{subsec:denseness-preliminary-results}. We thank the anonymous referee for suggesting that and for giving a lot of useful input.

Our main results concern the following question.

\noindent{\bf Question.} For a given form $F$ over a Dedekind domain $T$ with the fraction field $L$, investigate whether $R(F(T))$ is dense in infinitely many (or even more precisely, for almost all) $L_\p$, completions of $L$ with respect to nonzero prime ideals $\p$ of $T$, and similar questions. We first start with the cases with the positive answer.

\begin{thm}\label{thm-7}
Let $T$ be a Schur-Dedekind domain (see the definition and examples in \S\ref{subsec:Schur-Dedekind}) with the fraction field $L$, and $F$ an integral form in $n$ variables of degree greater than one. Suppose that there exist $a_1,a_2,\ldots,a_{n-1}\in T$ for which the polynomial $F(x,a_1,a_2,\dots,a_{n-1})$ in $x$ has non-zero discriminant. Then $R(F(T))$ is dense in  $L_{\p}$, the completion of $L$ with respect to a prime ideal $\p$, for infinitely many $\p$. 
\end{thm}

Recall that two forms $F_1$ and $F_2$ over a ring $T$ with the fraction field $L$ are said to be \emph{equivalent} if there is a non-singular linear transformation $A$ over $L$ such that $F_1(\overline{x})= F_2(A\overline{x})$. 
The \emph{order} $\text{o}(F)$ of a form $F$ is the smallest integer $m$ such that $F$ is equivalent to a form that contains only $m$ variables explicitly. A form in $n$ variables is called \emph{non-degenerate} if $\text{o}(F)=n$. The order $\text{o}(F)$ has the crucial role in studying our more general version of the question.

\begin{thm}\label{thm-any-degree-dense-from-some-p}
Fix a positive integer $d$. Let $T$ be a Dedekind domain with the fraction field $L$ such that for all nonzero prime ideals $\p$ we have that $T/\p=k_\p$ is a finite field. Let $F$ be a non-singular  form with coefficients in $T$ of degree $d$ such that $\o(F)\geq 3$. There exists a constant $M$ such that for all primes of norm $N(\p)>M$, we have that $R(F(T))$ is dense in $L_\p$.
\end{thm}

\subsubsection{Integral forms} 
Now, we will focus on integral forms. As an explicit example of Theorem~\ref{thm-any-degree-dense-from-some-p}, the following theorem gives a sufficient condition for the denseness of ratio sets of quintic forms. However, note that in this theorem, we do not assume that the form is non-singular at the cost that we require a higher order of a form.
\begin{thm}\label{thm-5}
Let $p$	be a prime. Let $F$ be an integral quintic form, whose reduction to the field $\F_p$ is a non-degenerate form of order at least $6$. If $p>101$, then $R(F(T))$ is dense in $\mathbb{Q}_p$.
\end{thm}

If we remove the assumption of non-singularity, it is clear that we need to require $\o(F)>\deg(F)$ as there are reducible (over $\overline{\Q}$) integral forms $F$ for which we have that $R(F)$ is not dense in $\Q_p$. 

\begin{thm}\label{thm-order-equals-degree-first-counterexamples}
Let $q>2$ be a prime number and denote $\zeta_q$ a primitive $q$th root of unity. Then for the following form 
$$F(x_0,\ldots,x_{q-2})=\prod_{\sigma\in\Gal(\Q(\zeta_q)/\Q)}\sigma(x_0+\zeta_q x_1+\cdots + \zeta_q^{q-2}x_{q-2})\in \Z[x_0,\ldots,x_{q-2}]$$
it holds that $\o(F)=q-1=\deg(F)$ and there are infinitely many prime numbers $p$ such that $R(F(T))$ is not dense in $\Q_p$.
\end{thm}

Even though it looks like that the non-singularity of $F$ is a strong condition in Theorem~\ref{thm-any-degree-dense-from-some-p}, it is crucial, even regardless of $\o(F)$. Of course, forms of the shape $G^d$ are trivial examples, but we construct more subtle examples, for the reason that will be clear later when we discuss our conjecture.

\begin{thm}\label{thm-composite-degree-counterexamples}
Let $n$ be a composite number, and write it as $n=qk$, where $q$ is a prime divisor of $n$. Let $m$ be a positive integer. Define the form 
$$F(x_0,\ldots,x_m)=x_0^{kq}+2(x_1^k+\cdots+x_m^k)^q\in \Z[x_0,\ldots,x_m].$$
Then $\o(F)=m+1$ and there are infinitely many prime numbers $p$ such that $R(F(T))$ is not dense in $\Q_p$ (regardless of the size of $m)$.
\end{thm}

When the degree of a form of high order is a fixed prime number, we do not have such counterexamples. We believe that the opposite is true, but we cannot prove it. We formulate a conjecture, for which we explain in \S\ref{subsec:our-conjecture} why we believe it is true. By Remark \ref{rmk:conjecture-holds-2,3,5}, the conjecture is already true for degrees 2, 3, and 5.

\begin{conj}\label{conj-prime-degree-all-cases}
Let $q$ be a prime number. There exists a positive integer $m$ (in terms of $q$) such that for any form $F\in\Z[x_0,\ldots,x_{m-1}]$ of order $\o(F)=m$ and $\deg(F)=q$, there is $M>0$ (in terms of $F$ and $m$) such that for all primes $p>M$, $R(F(T))$ is dense in $\mathbb{Q}_p$.
\end{conj}

In Conjecture~\ref{conj-prime-degree-all-cases}, it has to be $\o(F)>\deg(F)=q$ because otherwise we have counterexamples by a similar spirit as Theorem~\ref{thm-order-equals-degree-first-counterexamples}. Namely, the group $\Z/q\Z$ occurs as a Galois group over $\Q$ - there is such a subfield $L$ of $\Q(\zeta_p)$, where $p$ is any prime such that $q\mid p-1$. Then, $L$ is given by an irreducible polynomial $h$ of degree $q$ over $\Q$. By Chebotarev density theorem, $h$ stays irreducible modulo a positive density primes, and then we construct a counterexample with $\deg(F)=\o(F)=q$ in the same way as in Theorem~\ref{thm-order-equals-degree-first-counterexamples}.  

After our theorems, one could think about whether the existence of one $p$ such that $R(F(T))$ is dense in $\Q_p$ must imply that $R(F(T))$ is dense in infinitely many different fields $\Q_q$. We prove that this is incorrect. 

\begin{thm}\label{thm:dense-p-not-dense-almost-everywhere}
Let $n$ be a positive integer. There exists a form $F\in \Z[x_1,\ldots,x_n]$ (or a polynomial $F\in \Z[x]$ if $n=1$) for which there is a prime number $p$ and $M>0$, such that $R(F(T))$ is dense in $\Q_p$, but $R(F(T))$ is not dense in $\Q_q$ for all $q>M$.  
\end{thm}

We prove Theorem~\ref{thm:dense-p-not-dense-almost-everywhere} by explicitly constructing the required forms in Theorem~\ref{thm:dense-p-not-dense-almost-everywhere-explicit}.


\section{Preliminaries}

\subsection{Nonarchimedean complete discretely valued fields}\label{subsec:p-adic-numbers}
Let $(K,\nu)$ be a nonarchimedean complete discretely valued field, i.e., $K$ is complete with respect to the nonarchimedean metric induced by the absolute value coming from a surjective discrete valuation $\nu\colon K\longrightarrow \Z \cup \{\infty\}$ (or if we do not want to discuss normalisations of valuations, then we can require that the codomain is an abelian group isomorphic to $\Z$ together with $\infty$). Its ring of integers is $R=\{x\in K\;\colon \; \nu(x)\geq 0\}$, its invertible elements are $R^*=\{x\in K\;\colon \; \nu(x)= 0\}$, and its unique maximal ideal is $\p= \{x\in K\;\colon \; \nu(x) > 0\}$. For any $\pi \in R$ with $\nu(\pi)=1$, we have $\p=(\pi)$ and call $\pi$ a uniformiser. Let $k$ be the residue field of $(K,\nu)$, i.e., $R/\p$. For any element or a form with coefficients in $R$, say $f$, we will denote by $f_\p$ its image in the reduction modulo $\p$ (for forms, reducing their coefficients modulo $\p$), but recall that for the field $K$, $K_\p$ denotes the completion of $K$ with respect to $\p$. In this section, we will use the notation established here. 

Well-known examples are $\Q_p$, with the $p$-adic valuation $\nu_p$, and absolute value that defines the metric $|a|_p=p^{-\nu_p(a)}$, ring of integers $\Z_p$ and a residue field $\F_p$. Finite extensions of $\Q_p$ are examples too, but we need to normalise the valuation taken as an extension of $\nu_p$ to make the codomain $\Z$ (plus infinity) and not some fractional ideal of $\Z$. Another example is $K=L(\!(t)\!)$, the field of Laurent series over a field $L$, with a $t$-adic valuation $\nu_t$, and an absolute value $|f|_t=c^{-\nu_t(f)}$, for any $c>1$, for simplicity, let us fix $c=2$.
We now state Hensel's lemma since it is crucial for our results. 

\begin{thm}\cite[Consequence of II(4.6)]{Neukirch}
Let $f(x)\in R[x]$. Suppose that $\alpha\in k$ is a simple root of $f_\p(x)$. There exists a unique $A\in R$ such that $f(A)=0$ and $A\equiv \alpha\pmod{\p}$.
\end{thm}

\subsection{Useful results for studying denseness of forms in nonarchimedean complete discretely valued fields}\label{subsec:denseness-preliminary-results}

We survey the known results with direct generalisations stated here, especially to general nonarchimedean complete discretely valued fields. However, there is an entirely new result, when we consider the field of Laurent series over a field of characteristic zero, to our best knowledge, not considered in this context, see Proposition~\ref{prop:Q(t)-counterexample}.

We start with an immediate generalisation of \cite[Lemma 1]{miska}.

\begin{lem}\label{lem1}
Let $T$ be a dense subset of $R$. Let $f\colon R^n\rightarrow K$ be a continuous function. Then the following conditions are equivalent:
\begin{enumerate}
\item[(i)] $R(f(T))$ is dense in $K$.
\item[(ii)] $R(f(R))$ is dense in $K$.
\end{enumerate}
\end{lem}

\noindent{\bf Convention.} After Lemma~\ref{lem1}, it is handy to simplify the notation. Let $T$ be a dense subset of $R$, and $f\colon T^n\rightarrow K$ (which extends to $f\colon R^n\rightarrow K$) be a continuous function. Since $R(f(T))$ is dense in $K$ if and only if $R(f(R))$ is dense in $K$, we will in the remaining text focus on the latter case, and the $R(f)$ will denote $R(f(R))$. In applications, we will mostly think of $\N$ or $\Z$ as a subset of $\Z_p$ and even though there are infinitely many choices for the prime $p$, i.e., for $\Z_p$, it is always clear from the context that when we say $R(f)$ is dense in $\Q_p$, that we mean $R(f(\Z_p))$.\\

One more easy but very useful topological statement generalising \cite[Lemma 2.1]{garciaetal} is the following.

\begin{lem}\label{lem:dense-has-all-valuations}
If $S\subseteq K$ is dense in $K$, then for each $z\in\Z$, there is $s\in S$ with $\nu(s)=z$.
\end{lem}

We now present a few necessary conditions for denseness. 

\begin{lem}\label{lem:from-Zp-solutions-to-denseness}
Let $f\colon R^r\rightarrow R$ be a continuous function. If for any $z\in R$, we can find $\overline{x}, \overline{y}\in R^r$ such that $ f(\overline{y})\neq 0$ and $f(\overline{x})/f(\overline{y})=z$
then $R(f)=K$. 
\end{lem}

\begin{proof}
The conditions imply that $R\subseteq R(f)$, and for $z\in K\setminus R$, we have $1/z\in R\setminus\{0\}$, so there are $\overline{x}, \overline{y}\in R^r$ such that $ f(\overline{x})/f(\overline{y})=1/z$ and $f(\overline{x})f(\overline{y})\neq 0$, so $f(\overline{y})/f(\overline{x})=z\in R(f)$.
\end{proof}


\begin{lem}\label{lem:no-zeros-not-dense}
Suppose that the residue field $k$ is a finite field. Let $f\colon R^r\rightarrow K$ be a continuous function. If $f(\overline{x})\neq 0$ for all $\overline{x}\in R^r$, then $R(f)$ is not dense in $K$.
\end{lem}

\begin{proof}
If $f(\overline{x})\neq 0$ for all $\overline{x}\in R^r$, then by continuity, there is a neighbourhood $U\subseteq K$ of $0$ such that $f(R^r)\cap U=\emptyset$, or equivalently, there is $C>0$ such that $\nu(f(\overline{x}))<C$ for all $\overline{x}\in R^r$. Since $R^r$ is compact (as $k$ is finite) and $f$ is continuous, it follows that $f(R^r)$ is compact, hence bounded in $K$, so there is $M$ (possibly negative) such that $\nu(f(\overline{y}))>M$ for all $\overline{y}\in R^r$. Then $\nu_p(f(\overline{x})/f(\overline{y}))<C-M$ for $\overline{x}, \overline{y}\in R^r$, and hence $R(f)$ is not dense in $K$ by Lemma~\ref{lem:dense-has-all-valuations}.
\end{proof}

\begin{rem}
If, we have $f\colon R^r\rightarrow R$ instead of $f\colon R^r\rightarrow K$ in Lemma~\ref{lem:dense-has-all-valuations}, then we do not need the assumption that $k$ is a finite field, as in this case $\nu(f(\overline{x}))\geq 0$ for all $\overline{x}\in R^r$. However, in the setting of Lemma~\ref{lem:no-zeros-not-dense}, the condition on $k$ being finite is essential as we prove in the following proposition.
\end{rem}

\begin{prop}\label{prop:Q(t)-counterexample}
Let $K=\Q(\!(t)\!)$ and $R=\Q[\![t]\!]$. Consider $f\in R$  given as $f=\sum_{n\geq 0}a_n(f)t^n$. The map $\phi\colon R\longrightarrow K$ given by
$$
\phi(f)=\begin{cases}
			f, & \text{if $a_0(f)\in \Q\setminus\Z$;}\\
            t^{-a_0(f)}\cdot f, & \text{if $a_0(f)\in \Z\setminus\{0\}$;}\\
            t^{\lfloor a_1(f)\rfloor}, & \text{if $a_0(f)=0$.}\\
		 \end{cases}
$$
Then $\phi$ is continuous, $0\notin \phi(R)$, and $R(\phi)=K\setminus\{0\}$, so $R(\phi)$ is dense in $K$.
\end{prop}

\begin{proof}
We will prove that $\phi$ is continuous by finding a neighbourhood of each $r\in R$ on which $\phi$ is continuous. For $f$ such that $a_0(f)\in \Q\setminus\{0\}$, it is enough to use $U=\{g\in R\;\colon\; a_0(g)=a_0(f)\}$. For $f$ with $a_0(f)=0$, we consider $U=\{g\in R\;\colon\;a_0(g)=0,\; a_1(g)=a_1(f)\}$.  These correspond to the open balls centred at $f$ with radii, e.g., $1$ and $1/2$, respectively. Then, the function on these neighbourhoods is either a constant, identity, or multiplication by the constant, and all three maps are continuous (we can even see that in all cases there is a constant $C$ such that $|\phi(a)-\phi(b)|_t\leq C|a-b|_t$, for all $a,b\in U$, which implies the continuity of $\phi$ on $U$ immediately). 

It is clear that $0\notin \phi(R)$, because by the construction, when $a_0(f)\neq 0$, then $\phi(f)\neq 0$, and none of the functions with $a_0(f)=0$ is mapped to 0.

Now we prove $R(\phi)=K\setminus\{0\}$. We can write any $f\neq 0$ as $f=t^ng$, with $a_0(g)\neq 0$ and $n\in\Z$. If $a_0(g)\in \Q\setminus \Z$, then $\phi(g)/\phi(-nt)=g/t^{-n}=f$. If $a_0(g)\in \Z$, then $\phi(g)/\phi((-n-a_0(g))t)=gt^{-a_0(g)}/t^{-n-a_0(g)}=f$. 

Hence, since $t^n \in R(\phi)$, and $\lim_{n\rightarrow+\infty}t^n=0$, $R(\phi)$ is indeed dense in $K$.

\end{proof}

\begin{rem}\label{rmk:use-only-for-polynomials}
Note that Lemma \ref{lem:no-zeros-not-dense} cannot be used in the context of forms. We have seen that for polynomials, the condition of $k$ being finite is crucial. For example, this lemma is a very important criterion when studying polynomials (in one variable) over $\Q_p$.
\end{rem}

Now we mention one direct criterion for $R(F)$ being not dense in $K$, which is a direct generalisation of \cite[Theorem 1.8]{diagonal}, and it follows from Lemma~\ref{lem:dense-has-all-valuations}. 

A form $F$ is said to be \emph{isotropic} over a field $K$ if there is a non-zero vector $\overline{x}\in K^k$ such that $F(\overline{x}) =0$. Otherwise, $F$ is said to be \emph{anisotropic} over $K$.

\begin{lem}\label{lem:anisotropic-mod-p}
Let $F$ be a form or a polynomial in one variable with coefficients in $R$. If $F$ is a polynomial in one variable, assume that $F(\overline{x})\not\equiv 0\pmod{\p}$ for all $\overline{x}\in R$, or if $F$ is a form in at least two variables, assume that $F$ is anisotropic over $k$. Then $R(F)$ is not dense in $K$. 
\end{lem}


The following results of Miska, Murru, and Sanna play an important role in our proofs. We will state it in the context of nonarchimedean complete discretely valued fields and not just in $\Q_p$. Their proof does not depend on the cardinality of the residue field $k$, so it extends more generally. 

\begin{thm}\label{thm:polynomial-coprime-multiplicities-roots}\cite[Theorem 1.2]{piotr}
	Let $f\colon R\rightarrow K$ be an analytic function with at least two distinct zeros in $R$ with coprime multiplicities. Then $R(f)$ is dense in $K$.
\end{thm}

\begin{thm}\label{thm6}\cite[Corollary 1.3]{piotr}
Let $f\colon R\rightarrow K$ be an analytic function which has a simple zero in $R$. Then $R(f)$ is dense in $K$.
\end{thm}

We now prove a lemma which is an application of Theorem~\ref{thm6}.

\begin{lem}\label{lem:simple-linear-factor}
Let $F$ be a form over $R$ with at least two variables. Assume that $F$ has a simple linear factor over $R$. Then $R(F)$ is dense in $K$. 
\end{lem}

\begin{proof}
We may assume that $F$ is primitive, i.e., not all of its coefficients are in $\p$. Let $L(x_0,\ldots,x_n)=\sum_{i=0}^n a_ix_i$ be a simple factor of $F$, for which we may assume that it is simply $x_0$ (otherwise there is a linear change of coordinates transforming $L$ to $x_0$ and it can be made invertible even over $R$, since $F$ is primitive). Then $F=x_0G$, where $G(0,x_1,\ldots,x_n)\neq 0$ is not a zero polynomial, so it attains a nonzero value. Suppose that for $(b_1,\ldots,b_n)\in R^n$, we have $G(0,b_1,\ldots,b_n)\neq 0$. Then, $f(x_0)=F(x_0,b_1,\ldots,b_n)$ is a polynomial which has a simple root $x_0$ over $R$, so by Theorem \ref{thm6}, to conclude that $R(f)$ is dense in $K$. The statement follows from $R(f)\subseteq R(F)$.
\end{proof}

Very similarly, we have the following statement.

\begin{lem}\label{lem:form-simple-zero-Fp}
Let $F\in R[x_0,\ldots,x_{m-1}]$ be a form. Suppose that $F_\p$ has a non-singular zero over $k_\p$. Then $R(F)$ is dense in $K$.    
\end{lem}

\begin{proof}
There is $(a_0,\dots,a_{m-1})\in R^m$ such that $$F(a_0,\dots,a_{m-1})\equiv0\pmod{\p}, \;\text{and} \frac{\partial F}{\partial x_j}(a_0,\dots,a_{m-1})\not\equiv0\pmod{\p}$$ for some $j$ satisfying $0\leq j\leq m-1$. Define $f(x):=F(a_0,\dots,x,\dots,a_{m-1})$, the polynomial formed by replacing $a_j$ by $x$. We have $f(a_j)\equiv 0\pmod{\p}$ and $f'(a_j)\not\equiv 0\pmod{\p}$. Therefore, by Hensel's lemma, there exists a simple root of $f$ in $R$. Hence, by Theorem \ref{thm6}, $R(f)$ is dense in $K$, and so is $R(F)\supseteq R(f)$. 
\end{proof}

An immediate corollary of this lemma is the following.

\begin{cor}\label{cor:simple-linear-factor-mod-p}
Let $F$ be an integral form. Assume that the reduction of $F$ modulo $\p$ has a simple linear factor. Then $R(F)$ is dense in $K$.  
\end{cor}


\begin{ex}
Let $F(x,y)=a_0x^d+a_1x^{d-1}y+\cdots+a_{d-1}xy^{d-1}+a_dy^d$ be an integral form of degree $d\geq 2$. Let $p$ be a prime number such that $p\mid a_d$ and $p\nmid a_{d-1}$. 

By the assumption, it follows that $x$ is a simple zero of $F$ modulo $p$, so $R(F)$ is dense in $\Q_p$ by Corollary~\ref{cor:simple-linear-factor-mod-p}.

By symmetry $x\longleftrightarrow y$, we could also assume that $p\mid a_0$ and $p\nmid a_1$.     
\end{ex} 

\begin{ex}
For a prime $p$ and a polynomial $f(x)\in \mathbb{Z}[x]$, let $N_p(f(x))$ denote the number of zeros (including multiplicity) of $f(x)\equiv 0\pmod{p}$. If $N_p(f(x))=1$, then $f$ has a simple root modulo $p$, so by Lemma~\ref{lem:form-simple-zero-Fp}, $R(f)$ is dense in $\Q_p$. There are criteria by Sun \cite[Lemma 2.3, Theorem 5.3]{sun}, which imply that $N_p(f(x))=1$ for monic polynomials of degree 3 and 4, respectively, hence, for such polynomials, we have that $R(f)$ is dense in $\Q_p$.
\end{ex}

\subsection{Results on zeros of forms over finite fields}\label{subsec:results-about-zeros-over-Fq}

In our further work, we  use a few results over finite fields $\F_p$, which we can, using Hensel's lemma, lift to $\Q_p$ and obtain our results. For a prime $p$, let $q=p^r, r\geq 1$. Let $\mathbb{F}_q$ denote the finite field with $q$ elements. We consider forms $F$ in more variables and state a few known results which guarantee that $F$ has a zero over $\F_q$ of large enough cardinality. 

\begin{thm}\cite[Theorem 1]{Lewis-Schuur}\label{thm-Lewis-Schuur}
Let $G$ be an absolutely irreducible form
of degree $f$ over $\F_q$ and $H$ a form of degree $e$ over $\F_q$ not divisible by $G$. There exists a function $A(f, e)$ such that if 
$q > A(f, e)$, then there exists an $\F_q$-point which is a non-singular zero of $G$
and which is not a zero of $H$. As a consequence, if $F=GH$, then there is a non-singular zero of $F$ in $\F_q$.   
\end{thm}

\begin{rem}\label{rmk:explicit-bounds-p}
If we want an explicit version of Theorem~\ref{thm-Lewis-Schuur}, we can use \cite[Theorem 2]{Zahid}.  
\end{rem}

We have concrete values for forms of degree 5 in the following theorem.

\begin{thm}\label{cor-8}\cite[Corollary 4.5]{leep}
Let $f$ be a non-degenerate quintic form of order at least $6$ over $\mathbb{F}_q$. If $q>101$, then $f$ has a non-singular $\mathbb{F}_q$-rational zero.
\end{thm}

\subsection{Schur-Dedekind domains}\label{subsec:Schur-Dedekind}

In this section, we state and prove Schur's theorem for number fields and global fields. The proofs presented here were provided by an anonymous referee and all the credit for this section goes to the referee. We are grateful to the referee for these general results, and introducing the new notion of Schur-Dedekind domains to us, which have enabled us to significantly broaden and generalize our work. 

\par 
Let $T$ be a Dedekind domain. Let $f\in T[x]$ be a polynomial of positive degree. Let $\mathfrak{p}$ be a nonzero prime ideal of $T$. We say that $\mathfrak{p}$ is a prime divisor of $f$ if there exists $r\in T$ such that $\mathfrak{p} \mid (f(r))$, where $(f(r))=Tf(r)$ is the principal ideal generated by $f(r)$.
\begin{defn}
Let $T$ be a Dedekind domain. We say that $T$ is a Schur-Dedekind domain if $T$ satisfies the following two conditions.
\begin{enumerate}
\item $T$ has infinitely many nonzero prime ideals.
\item Every polynomial in $T[x]$ of positive degree has infinitely many
prime divisors.
\end{enumerate}
\end{defn}

\begin{thm}\label{Schur-Dedekind-1}
    Let $T$ be a Schur-Dedekind domain and let $L$ be the fraction field of $T$. Let $M$ be a finite extension of $L$ and let $S$ be the integral closure of $T$ in $L$. Then $S$ is a Schur-Dedekind domain.
\end{thm}
\begin{proof}
Suppose that $[M:L]=n$ and let $\{v_1,\ldots,v_n\}$ be a $L$-basis of the vector space $M$. Then $\{v_1,\ldots,v_n\}$ is also a $L(x)$-basis of the vector space $M(x)$ because $[M(x):L(x)]=n$.
We define the norm maps
\[
N_{M/L}:M^\times \to L^\times
\quad \text{and} \quad
N_{M(x)/L(x)}:M(x)^\times \to L(x)^\times
\]
in the following standard way.
Let $\alpha\in M(x)^\times$ and suppose that
\[
\alpha v_j=\sum_{i=1}^{n} a_{ij}v_i,
\qquad 1\le j\le n,
\]
where $a_{ij}\in L(x)$ for $1\le i,j\le n$. Let $A=(a_{ij})$
be the corresponding $n\times n$ matrix with entries in $L(x)$. Then $
N_{M(x)/L(x)}(\alpha)=\det(A).$
The norm map $N_{M/L}$ is defined similarly. We see that
\[
N_{M/L}
=
N_{M(x)/L(x)}\big|_{M^\times},
\]
the restriction of $N_{M(x)/L(x)}$ to $M^\times$.

\par 

Suppose that $f\in M[x]$, $f\neq 0$. Then each $a_{ij}\in L[x]$, and if $a_{ij}\neq 0$, then $\deg(a_{ij})\leq \deg(f)$.
It follows that $g=N_{M(x)/L(x)}(f)\in L[x]$
and $\deg(g)=n\deg(f)$, because if $0\neq b\in M$ is the coefficient of $x^{\deg(f)}$ in $f$, then the coefficient of $x^{n\deg(f)}$ in $g$ is
$N_{M/L}(b)\neq 0$.
\par 
We observe that if $\alpha\in M$, $\alpha\neq 0$, then
\[
g(\alpha)
=
N_{M(x)/L(x)}\bigl(f(\alpha)\bigr)
=
N_{M/L}\bigl(f(\alpha)\bigr).
\]

We also observe that the following statements hold:

\begin{enumerate}
\item If $\alpha\in S$, then $N_{M/L}(\alpha)\in T$.

\item If $\alpha$ is a unit in $S$, then $N_{M/L}(\alpha)$ is a unit in $T$.

\item If $f\in S[x]$, then
\[
N_{M(x)/L(x)}(f)\in T[x].
\]
\end{enumerate}

It is known that $S$ is a Dedekind domain, and since $S$ is integral over $T$ and $T$ has infinitely many nonzero prime ideals, the same holds for $S$.
\par 

Let $f\in S[x]$ and assume that $f$ has positive degree. If $f(0)=0$, then $x\mid f$, and $f$ has infinitely many prime divisors because if $\mathfrak{p}$ is a nonzero prime ideal in $S$, and if $\mathfrak{p}\mid a$, then $\mathfrak{p}\mid f(a)$. Now assume that
$f(0)=c\neq 0$, where $c\in S$. Suppose that $\mathfrak{p}_1,\ldots,\mathfrak{p}_n$ are the only prime divisors of $f$, where $n\geq 0$. By the weak approximation theorem in $S$, there exists $a\in S$ such that $\nu_{\mathfrak{p}_i}(a)=1$ for $1\leq i\leq n$.
Then $f(acx)=c\,h(x)$, where $h(x)=1+c_1x+c_2x^2+\cdots \in S[x]$
and $\nu_{\mathfrak{p}_i}(c_j)>0$ for all $1\leq i\leq n$ and all $j\geq 1$.
A prime divisor of $h$ would also be a prime divisor of $f$ and would be distinct from $\mathfrak{p}_1,\ldots,\mathfrak{p}_n$. Thus, to obtain a contradiction, it remains to show that $h$ has a prime divisor.

\par 

Since $h$ has positive degree, it follows that
$N_{M(x)/L(x)}(h)=g\in T[x]$, and $g$ has positive degree because
$\deg(g)=n\deg(h)$. Since $T$ is a Schur-Dedekind domain, there exists $r\in T$ such that $g(r)$ has a prime divisor and thus $g(r)$ is not a unit in $T$. Then
$N_{M(x)/L(x)}\bigl(h(r)\bigr)=g(r)$
is not a unit. It follows that $h(r)$ is not a unit in $S$, and thus $h$ has a prime divisor in $S$. 
\end{proof}

\begin{thm}\label{Schur-Dedekind-2}
Let $L$ be a number field with ring of integers $T$. Then $T$ is a Schur-Dedekind domain.
\end{thm}
\begin{proof}
The ring of integers $T$ of a number field $L$ is the integral closure of $\mathbb{Z}$ in $L$. It is known that $T$ is a Dedekind domain. The result now follows from Schur's theorem for $\Z$, see \cite{Schur}, and Theorem~\ref{Schur-Dedekind-1}.
\end{proof}
\begin{thm}\label{Schur-Dedekind-3}
Let $F$ be a field and let $T=F[t]$. Then $T$ is a Schur--Dedekind domain.
\end{thm}
\begin{proof}
One can prove that $T$ has infinitely many nonzero prime ideals using the method of the standard Euclidean proof of the same fact for $\mathbb{Z}$.
To prove that a polynomial $f\in F[t]$ of positive degree has infinitely many prime divisors, one can follow Schur's original proof for $\mathbb{Z}$ and use the degree function in place of the absolute value on $\mathbb{Z}$.
\end{proof}
Recall that a \emph{global field} is a field that is a finite extension of either $\mathbb{Q}$ or $\mathbb{F}_q(t)$, where $\mathbb{F}_q$ is the finite field with $q$ elements.

\begin{cor}\label{Corollary-Schur-Dedekind-1}
Let $T$ be the ring of integers of a global field. Then $T$ is a Schur-Dedekind domain.
\end{cor}
\begin{proof}
This follows from Theorem~\ref{Schur-Dedekind-1}, Theorem~\ref{Schur-Dedekind-2}, and Theorem~\ref{Schur-Dedekind-3}.
\end{proof}

\begin{ex}
Here is an example of a PID (and thus a Dedekind domain) containing infinitely many nonzero prime ideals that is not a Schur-Dedekind domain. Let $S$ be the set of prime numbers $p$ in $\mathbb{Z}$ such that either $p=2$ or $p\equiv 1 \pmod{4}$.
Let $T=\mathbb{Z}_S$, the localization of $\mathbb{Z}$ at $S$. Thus
$T=\mathbb{Z}[\{\frac{1}{p}\}]$, where $p$ runs through the prime numbers in $S$. Then $T$ is a PID because the localization of a PID is again a PID. The primes of $T$ are given by the rational primes $p$ where $p\equiv 3 \pmod{4}$.
\par 
Let $f(x)=x^2+1\in T[x]$. We now show that $f$ has no prime divisors, and thus $T$ is not a Schur-Dedekind domain. It is sufficient to show that if $\alpha\in T$, then $f(\alpha)$ is a unit in $T$.
\par 
Let $\alpha\in T$. Then $\alpha=\displaystyle \frac{a}{b}$,
with $a,b\in\mathbb{Z}$, $\gcd(a,b)=1$, and every prime divisor of $b$ (if any) lies in $S$. Since
\[
f\!\left(\frac{a}{b}\right)
=
\frac{a^2+b^2}{b^2},
\]
and $\gcd(a,b)=1$, it follows that every prime divisor of $a^2+b^2$ also lies in $S$. Thus $f(\alpha)$ is a unit in $T$, as desired.
\end{ex}

\section{Hypersurfaces, irreducibility, smoothness}\label{sec:ag-cones}

In this section, we introduce well-known tools that provide a new perspective that is useful for proving statements related to the denseness in $\Q_p$ of forms of arbitrary (large) order. We believe that most of this section is already known, except Theorem~\ref{thm:cone-over-Fp-p-large-over-Q}, but we keep the full section to improve our exposition. Since forms are defined as homogeneous polynomials, we can apply some methods and results from commutative algebra and algebraic geometry. Namely, each form gives a projective variety --- a hypersurface.  We will define our objects over an arbitrary ring \(T\), whose field of fractions is $L$, which can be chosen 
among
 \(\mathbb{Z},\mathbb{Z}_p,\F_p,\mathbb{Q},\Q_p,\mathbb{C}\). For homogeneous forms $F_1,\ldots,F_k\in T[x_0,\ldots,x_n]$ denote their zero set by
$$V_+(F_1,\ldots,F_k)=\{(x_0:\ldots:x_n)\in \P_T^n \colon F_1(x_0:\ldots:x_n)=\cdots=F_k(x_0:\ldots:x_n)=0\}.$$

We start with a simple technical, but useful lemma.

\begin{lem}\label{lem:more-variables-nontrivial-solution}
    Assume we are given \(m,n\in \mathbb{N}\) such that \(m\leq n\). Let us have homogeneous \(f_1,f_2,\ldots,f_m\in k[x_0,x_1,\ldots,x_n]\) over an algebraically closed field \(k\). Then there exists a point \(a\coloneq (a_0:\ldots:a_n)\in \P^n_k\) such that
    \[f_1(a)=f_2(a)=\cdots=f_m(a)=0.\]
\end{lem}
\begin{proof}
    The proof goes by the classical dimension theory. In particular, we inductively apply \cite[Thm. I.7.2]{hartshorne}. Each \(f_i\) gives one projective hypersurface \(V_+(f_i)\subseteq \mathbb{P}^n_k\). Therefore, we have
    \[\dim(V_+(f_1,f_2))=\dim V_+(f_1)\cap V_+(f_j)\geq n-1+n-1-n=n-2.\]
    By applying the theorem \(m-1\) times (to $\dim(V_+(f_1,f_2,\ldots,f_s)$, $2\leq s\leq m$), we obtain
    \[\dim \bigcap_{i=1}^m V_+(f_i)\geq n-m\geq 0.\]
    All varieties of dimension at least \(0\) over an algebraically closed field are non-empty. 
    
\end{proof}

\begin{defn}
We say that a form $F\in T[x_0,\ldots,x_n]$ (or the variety $V_+(F)=0$ defined by $F$) is \emph{non-singular} or \emph{smooth} if $V_+(\frac{\partial F}{\partial x_0},\ldots,\frac{\partial F}{\partial x_n})=\emptyset$ considered over the algebraic closure $\overline{L}$. Otherwise, it is called \emph{singular}.
\end{defn}


\begin{lem}\label{lem:cone_partial}
Let \(F\in L[x_0,\ldots,x_n]\) be a form. If $\o(F)<n+1$, there exists $(\lambda_0,\ldots,\lambda_n)\neq (0,\ldots,0)\in L^{n+1}$ such that \(\sum_{i=0}^n\lambda_i\frac{\partial F}{\partial x_i}=0\). The converse holds when $\char(L)=0$ or if $\char(L)>0$ and all for monomials $x_0^{e_0}\cdots x_n^{e_n}$ appearing in $F$, we have that $\char(L)\nmid \max(e_0,1)\cdots \max(e_n,1)$.
\end{lem}

\begin{rem}
The notation $\max(e_i,1)$ is used to disregard the variables that do not appear in the monomials, i.e., when $e_i=0$, which would make the divisibility condition useless.
\end{rem}

\begin{proof}
Assume that $\o(F)<n+1$. Then, there is a non-negative integer $s<n$, $\alpha_{ij}\in L$, for $0\leq i\leq s$ and $0\leq j\leq n$ and a form $G\in L[y_0,\ldots,y_s]$ such that  
$$F(x_0,\ldots,x_n)=G\left(\sum_{j=0}^n\alpha_{0j}x_j,\ldots,\sum_{j=0}^n\alpha_{sj}x_j\right).$$
Then, for $(\lambda_0,\ldots,\lambda_n)\in L^{n+1}$ we compute
$$\sum_{i=0}^n\lambda_i\frac{\partial F}{\partial x_i}=\sum_{j=0}^s\left(\sum_{i=0}^n\lambda_i\alpha_{ji}\right)\frac{\partial G}{\partial y_j}.$$

Now the system of equations $\sum_{i=0}^n\lambda_i\alpha_{ji}=0$ for $0\leq j\leq s$ has $s+1$ equation in $n+1$ unknown variables $(\lambda_0,\ldots,\lambda_n)$, so it has a non-trivial solution. For such a solution $(\lambda_0,\ldots,\lambda_n)$, we have that $\sum_{i=0}^n\lambda_i\frac{\partial F}{\partial x_i}=0$.

Assume there is $(\lambda_0,\ldots,\lambda_n)\neq (0,\ldots,0)\in L^{n+1}$ such that \(\sum_{i=0}^n\lambda_i\frac{\partial F}{\partial x_i}=0\). There is an invertible matrix $A=(\alpha_{ij})_{0\leq i,j\leq n}$ over $L$ such that $\alpha_{i0}=\lambda_i$ because $(\lambda_0,\ldots,\lambda_n)\neq (0,\ldots,0)$. Then, for $G(X)=F(AX)$, we have 
$$\frac{\partial G}{\partial x_0}=\sum_{i=0}^n \lambda_i\frac{\partial F}{\partial x_i}=0,$$
which under the assumptions of the lemma, implies that $G$ is a form that does not contain $x_0$, so $\o(F)<n+1$.
\end{proof}

There are three direct consequences of Lemma~\ref{lem:cone_partial}.

\begin{cor}\label{cor:order-the-same-when-raised-to-powers}
    Let $r$  be a positive integer and \(F\in L[x_0,x_1,\ldots,x_n]\), where $\char(L)=0$, be a form with $\o(F)=n+1$. Then we have  \(\o(F)=\o(F^r). \)
\end{cor}

\begin{proof}
Let $\lambda_0,\ldots,\lambda_n\in L$ such that $\sum_{i=0}^n\lambda_i\frac{\partial F^r}{\partial x_i}= 0$. Then, by the chain rule, we have $\sum_{i=0}^n\lambda_i\frac{\partial F}{\partial x_i}rF^{r-1}=0$ and since $\char(L)=0$, it implies $\sum_{i=0}^n\lambda_i\frac{\partial F}{\partial x_i}= 0$. As $\o(F)=n+1$, we conclude $\lambda_0=\cdots=\lambda_n=0$, so, by Lemma~\ref{lem:cone_partial}, it follows that $\o(F^r)=n+1=\o(F)$.
\end{proof}

\begin{cor}\label{cor:singular}
Let \(F\in L[x_0,\ldots,x_n]\) be a form. If $\o(F)<n+1$, then $F$ is singular. 
\end{cor}

\begin{proof}
By Lemma~\ref{lem:cone_partial}, there is $(\lambda_0,\ldots,\lambda_n)\neq (0,\ldots,0)\in L^{n+1}$ (we may assume $\lambda_0\neq 0$) such that \(\sum_{i=0}^n\lambda_i\frac{\partial F}{\partial x_i}=0\). We apply Lemma~\ref{lem:more-variables-nontrivial-solution} to conclude $V_+(\frac{\partial F}{\partial x_1},\ldots,\frac{\partial F}{\partial x_n})\neq \emptyset$ because we have $n$ equations in $n+1$ variables. Since $\frac{\partial F}{\partial x_0}=-\frac{1}{\lambda_0}\sum_{i=1}^n\lambda_i\frac{\partial F}{\partial x_i}$, we have that $V_+(\frac{\partial F}{\partial x_0},\ldots,\frac{\partial F}{\partial x_n})=V_+(\frac{\partial F}{\partial x_1},\ldots,\frac{\partial F}{\partial x_n})\neq \emptyset$.
\end{proof}

Let $T$ be a Dedekind domain with infinitely many prime ideals, with the fraction field $L$. For a $F$ be a form over $L$, by clearing the denominators, we may assume that $F$ has coefficients in $T$. The following result is essential for the investigation of our question.  We shall show that form \(F\) is of maximal order if and only if its reduction to \(k_\p=T/\p\), denoted $F_\p$, is of maximal order for all, except for finitely many, prime ideals \(\p\subseteq T\). Firstly, we need two preparatory statements.

\begin{cor}\label{cor:cone_modp}
Let \(F\in T[x_0,\ldots,x_n]\) be a form. If $\o(F)<n+1$, for all except finitely many prime ideals $\p\subseteq T$, $\o(F_\p)<n+1$. 
\end{cor}
\begin{proof}
Immediate consequence of Lemma~\ref{lem:cone_partial}, since derivation commutes with the base change modulo \(\p\). We avoid finitely many $\p$ for which $\nu_\p(\det A)>0$, where $A$ is the invertible matrix used to establish an isomorphism of $F$ with a form of smaller number of variables, and for which $\nu_p(\max(e_0,1)\cdots \max(e_n,1))>0$, for all $x_0^{e_0}\cdots x_n^{e_n}$ appearing as monomials in $F$.    
\end{proof}

We conclude this section with the theorem saying that a form of a maximal order over $T$ keeps the maximal order over almost all $k_\p$.

\begin{thm}\label{thm:cone-over-Fp-p-large-over-Q}
Assume $\char(L)=0$. Let \(F\in T[x_0,\ldots,x_n]\). Then \(\o(F)=n+1\) if and only if \(\o(F_\p)=n+1\) for all except finitely many prime ideals \(\p\subseteq T\).
\end{thm}

\begin{proof}

If $\o(F)<n+1$, then by Corollary~\ref{cor:cone_modp}, $\o(F_\p)<n+1$ for almost all $\p$, proving the implication ``$\o(F_\p)=n+1$ for almost all prime ideals $\p$ $\implies$ $\o(F)=n+1$''.

For the other direction, we assume that \(o(F)=n+
1\). Then, since $\char(L)=0$, we have that $\frac{\partial F}{\partial x_i}$ for $0\leq i\leq n$ are linearly independent over $L$. Let $\ell$ be the number of all monomials in $x_0,\ldots,x_n$ of degree $\deg(F)-1$, i.e., all monomials that could possibly appear in  $\frac{\partial F}{\partial x_i}$. For each $0\leq i\leq n$, let $v_i$ be the vectors whose coordinates are coefficients of the monomials appearing in $\frac{\partial F}{\partial x_i}$, which we identify with an $\ell$-dimensional $L$-vector space. The linear independence of $\frac{\partial F}{\partial x_i}$ over $L$ is equivalent to saying that the vectors $v_i$ are linearly independent over $L$ for $0\leq i\leq n$. Hence, there is an $(n+1)\times (n+1)$ submatrix $B$ of coordinates of $v_i$s with $\det(B)\neq 0$. Then, for only finitely many prime ideals $\p\subseteq T$, we have $\nu_p(\det(B))>0$. Hence, for all but finitely many $\p\subseteq T$, $\nu_p(\det(B))=0$, so $B$ is invertible in $k_\p=T/\p$, so $v_i$ for $0\leq i\leq n$ stay independent over $k_\p$. Hence, $\frac{\partial F}{\partial x_i}$ for $0\leq i\leq n$  are linearly independent over $k_\p$, so by Lemma~\ref{lem:cone_partial}, we conclude that $o(F_\p)=n+1$, and this happens for all but finitely many prime ideals $\p\subseteq T$.

\end{proof}

\section{Denseness (or not) of forms for infinitely many primes}

We now prove the main results of this paper. Recall that in this section, we consider the question of whether $R(F)$ is dense in infinitely many completions with respect to prime ideals $\p$ and its variations. 

\subsection{Denseness for infinitely many prime ideals $\p$}\label{subsec:dense-infinitely-p}

We start with a proof of Theorem \ref{thm-7}, which gives a criterion when $R(F)$ is dense in $L_\p$, a completion of the fraction field $L$ of a Schur-Dedekind domain $T$ with respect to infinitely many prime ideals $\p\subseteq T$.

\begin{proof}[Proof of Theorem \ref{thm-7}]
Suppose that the polynomial defined as $$f(x):=F(x,a_1,a_2,\dots,a_{n-1})$$ has a non-zero discriminant $D(f)$. Since $T$ is a Dedekind-Schur domain, there are infinitely many primes $\p \subseteq T$ such that $D(f)\notin \p$ and so that there is $n_\p\in T$ such that $f(n_\p)\in \p$.  Since $D(f)\not\equiv0\pmod{\p}$, $f$ has distinct roots modulo $\p$. Hence, $f'(n_\p)\not\equiv 0\pmod{\p}$. By Lemma~\ref{lem:form-simple-zero-Fp}, $R(f)$ is dense in $L_\p$, so $R(F)$ is also dense in $L_\p$. 
\end{proof}

\subsection{Denseness for all sufficiently large primes}\label{subsec:dense-all-large-p}

Now, we prove Theorem~\ref{thm-any-degree-dense-from-some-p}, which roughly states that, given that the form $F$ has enough variables, $R(F)$ is dense in $L_\p$, the completion of $L$ at all primes $\p$ of $T$ of sufficiently large norm. We need a preparatory lemma.

\begin{lem}\label{lem:spread-out}
    Assume that \(T\) is a Dedekind domain and \(L=\Frac{(T)}\). Assume that we have a non-singular homogeneous form \(F\) with coefficients in \(T\).  Then, for all except maybe finitely many \(\mathfrak{p}\in \Spec T\) we have that \(F_\mathfrak{p}\) is non-singular form over \(T/\mathfrak{p}\).
\end{lem}

\begin{proof}
The proof of this lemma is a direct application of  \cite[Theorem 3.2.1]{Poonen-book}. Since \(T\) is a Dedekind domain, \(\Spec T\) is a Noetherian integral scheme.  We are assuming smoothness of \(F\) over \(L\), which means \(X_L \coloneq X\times_T L\) is a smooth scheme over a field \(L\). Since we have a projective variety, morphism \(X\to \Spec T\) is by definition a morphism of finite presentation.  Therefore, we can apply statement \((ii)\) of the cited theorem to conclude there exists a dense open \(U\subseteq \Spec T\) such that morphism \(X_U\to U\) is smooth. Since smoothness is preserved under base change, \(X_\mathfrak{p}\coloneq X\times_T T/\mathfrak{p}\) is a smooth scheme over field \(T/\mathfrak{p}\) for all \(\mathfrak{p} \in U\).
 All closed sets in a Zariski topology are given by \(V(I)\) for some ideal \(I\subseteq T\). Denseness implies \(I\neq (0)\). Since \(T\) is a Dedekind domain, we have decomposition \(I=\mathfrak{p}_1^{e_1}\mathfrak{p}_2^{e_2}\ldots \mathfrak{p}_r^{e_r}\). Hence, \(U=V(\mathfrak{p}_1)^c \cap V(\mathfrak{p_2})^c \cap \ldots\cap V(\mathfrak{p}_r)^c\). In other words, for every other prime ideal \(\mathfrak{p}\in T\backslash\{\mathfrak{p}_1,\mathfrak{p}_2,\ldots, \mathfrak{p}_r\}\) we have that \(X_\mathfrak{p}=V(F_\mathfrak{p})\) is a smooth scheme over field \(T/\mathfrak{p}\).

\end{proof}

\begin{proof}[Proof of Theorem~\ref{thm-any-degree-dense-from-some-p}]

We may assume that $F$ does not have a (simple, implied by non-singularity) linear factor because then we can conclude by Lemma~\ref{lem:simple-linear-factor}.

By Theorem~\ref{thm:cone-over-Fp-p-large-over-Q}, there is $C>0$ such that for all $N(\p)>C$, we have $\o(F)=\o(F_\p)\geq 3$.

By Lemma~\ref{lem:spread-out}, there is $C'>0$ such that for all $N(\p)>C'$, we have that $F$ reduced modulo $\p$, $F_\p$, is smooth over $k_\p$.  Now we use the smoothness of $F_\p$ over $k_\p$ for $N(\p)>\max(C,C')$ to prove that it is absolutely irreducible.

\begin{lem}\label{lem:not-absirred-not-smooth}
Assume that $F\in L[x_0,\ldots,x_{m-1}]$, where $L$ is a field, is a form of degree $d>1$ such that $\o(F)\geq 3$ is reducible over the algebraic closure of $L$. Then, $F$ is singular.  
\end{lem}

\begin{proof}
Write $F=GH$. Then, by Lemma~\ref{lem:more-variables-nontrivial-solution}, that $V\coloneq V_+(G,H)\neq \emptyset$ because we assume that $\o(F)\geq 3$. Now, $F$ is singular in any $P\in V$ because $\frac{\partial F}{\partial  x_i}(P)=\frac{\partial G}{\partial  x_i}(P)H(P)+\frac{\partial H}{\partial  x_i}(P)G(P)=0$.
\end{proof}

Now we continue the proof.  Since $F_\p$ is absolutely irreducible, we apply Theorem~\ref{thm-Lewis-Schuur} to $F_\p$, which implies that there is $A=A(d,0)$ for which we know that if $N(\p)>M\coloneq \max(C,C',A)$, $F$ has a non-singular zero in $k_\p$. Now, we conclude by Lemma~\ref{lem:form-simple-zero-Fp}.
\end{proof}

\begin{rem}\label{rmk:effectiveness}
When the integral form $F$ is given, then we can get a concrete value of $M$ in Theorem~\ref{thm-any-degree-dense-from-some-p}. Namely, we can determine the bound $C'$ such that $p>C'$ implies that $F_p$ is non-singular over $\F_p$, and we use the effective theorem of Zahid, see Remark~\ref{rmk:explicit-bounds-p} for the other part of the bound for $M$. As we noted, since for $p>C'$ we have that $F_p$ is non-singular over $\F_p$, so automatically it has order $\o(F_p)=\o(F)$, so we do not need Theorem~\ref{thm:cone-over-Fp-p-large-over-Q}.  
\end{rem}

In the same way, we give a proof of Theorem \ref{thm-5}.
\begin{proof}[Proof of Theorem \ref{thm-5}]
	By Theorem \ref{cor-8}, $F$ has a non-singular zero in $\F_p$, and then we conclude as in the end of the proof of Theorem~\ref{thm-any-degree-dense-from-some-p}.
\end{proof}


We now recall that the condition of smoothness is necessary, as otherwise the statement of Theorem~\ref{thm-any-degree-dense-from-some-p} is not true. We start with first counterexamples.

\begin{proof}[Proof of Theorem~\ref{thm-order-equals-degree-first-counterexamples}]
We first prove that $\o(F)=q-1$. Consider any of $\sigma(x_0+\zeta_q x_1+\cdots + \zeta_q^{q-2}x_{q-2})$ for $\sigma\in\Gal(\Q(\zeta_q)/\Q)$. Since $1,\zeta_q,\ldots,\zeta_q^{q-2}$ are linearly independent over $\Q$, any invertible linear change of variables over $\Q$ $(x_0,\ldots,x_{q-2})=A(y_0,\ldots,y_{q-2})$ will keep all $q-1$ distinct variables $y_0.\ldots,y_{q-2}$. Then, their product will contain monomials $y_k^{q-1}$ with a non-zero coefficient for all $0\leq k\leq q-2$, hence $\o(F)=q-1$.   

Now consider any $p$ such that $p$ is a primitive root modulo $q$. There are infinitely many of such primes by Dirichlet's theorem on primes in arithmetic progressions \cite[VII(5.14)]{Neukirch} (fix one residue $a$ modulo $q$ which is a primitive root, then $a+dq$ has infinitely many primes as $\gcd(a,d)=1$). Then, it is well known that $[\F_p(\zeta_q):\F_p]=q-1$, since $p$ is then inert in $\Q(\zeta_q)$ (which can be deduced easily from, e.g., \cite[I(10.3)]{Neukirch} and factorisation of primes in number fields), i.e., $1,\zeta_q,\ldots,\zeta_q^{q-2}$ are $\F_p$-linearly independent, so $F(x_0,\ldots,x_{q-2})=0$ in $\F_p$ implies that $x_0=\cdots=x_{q-2}=0$. So, if $p\mid F(x_0,\ldots,x_{q-2})$, then $p\mid x_0,\ldots,p\mid x_{q-2}$. This implies that $q-1\mid \nu_p(F(x_0,\ldots,x_{q-2}))$ for any $(x_0,\ldots,x_{q-2})\in \Z_p^{q-1}$. Therefore, $R(F)$ is not dense in $\Q_p$ by Lemma~\ref{lem:dense-has-all-valuations}.
\end{proof}

Now we prove Theorem~\ref{thm-composite-degree-counterexamples} about more subtle counterexamples.

\begin{proof}[Proof of Theorem~\ref{thm-composite-degree-counterexamples}]
We prove that $\o(F)=m+1$ using Lemma~\ref{lem:cone_partial}. Namely, the partial derivatives of $F$ are
$$\dfrac{\partial F}{\partial x_0}=kqx_0^{kq-1},\;\text{and for $1\leq i\leq n$,}\;\dfrac{\partial F}{\partial x_i}=2kqx_i^{k-1}(x_1^k+\ldots+x_m^k)^q.$$
It is clear that $\frac{\partial F}{\partial x_0}$ is linearly independent from the others, and it is straightforward to see that $\frac{\partial F}{\partial x_i}$ for $1\leq i\leq n$ are linearly independent too.

Now, we consider only primes $p\nmid kq$. There are infinitely many primes $p$ such that $-2$ is not a $q$th power modulo $p$. This follows from Chebotarev density theorem as this condition is equivalent to $p$ being split completely in the field $L\coloneq \Q(\sqrt[q]{-2},\zeta_q)$, where $\zeta_q$ is a $q$th root of unity. We know that $L$ is Galois over $\Q$ of degree $q(q-1)$. Primes $p$ for which $q\nmid p-1$ cannot be split completely, so we only focus on primes $p$ so that $q\mid p-1$. They have density $\frac{1}{q-1}$. However, Chebotarev density theorem gives that the density of primes splitting completely in $\Q(\sqrt[q]{-2},\zeta_q)$ is $\frac{1}{q(q-1)}<\frac{1}{q-1}$, see, for example, \cite[VII (13.6)]{Neukirch}, so there are infinitely many such $p$ which do not split in $\Q(\sqrt[q]{-2},\zeta_q)$.

For all such $p$, $-2$ is not a $q$th power in $\F_p^*$, so $p\mid F(x_1,\ldots,x_m)$ implies that $p\mid x_1$ and $p\mid x_2^k+\cdots+x_m^k$. This implies that $q\mid \nu_p(F(x_1,\ldots,x_m))$, so $R(F)$ cannot be dense in $\Q_p$ by Lemma~\ref{lem:dense-has-all-valuations}.
\end{proof}

\begin{rem}
As we see in the proof, it is crucial that $k>1$ because this implies $o(F)=m+1$. If $k=1$, then the partial derivatives are not independent, which is also obvious since the order of the form $x_1+\cdots+x_m$ is 1.
\end{rem}

\subsection{Our conjecture}\label{subsec:our-conjecture}

We immediately see that the forms from Theorem~\ref{thm-composite-degree-counterexamples} do not work for prime degrees because that would bound the order of such forms. So, we mention the cases in which Conjecture~\ref{conj-prime-degree-all-cases} is true.

\begin{prop}\label{prop:F-irreducible}
If $F$ from Conjecture~\ref{conj-prime-degree-all-cases} is absolutely irreducible, then Conjecture~\ref{conj-prime-degree-all-cases} is true in this case.
\end{prop}

\begin{proof}
If $F$ is not absolutely irreducible, since $\deg(F)$ is prime, it has to be a product of linear forms, and exactly $\deg(F)$ of them, so it implies that $\o(F)\leq \deg(F)$, which is a contradiction. Hence, $F$ is absolutely irreducible and then the statement follows from Theorem~\ref{thm-Lewis-Schuur} and the end of the proof of Theorem~\ref{thm-any-degree-dense-from-some-p}.  
\end{proof}

So, we may assume that $F$ is reducible. Now, we have one more easy criterion in which  Conjecture~\ref{conj-prime-degree-all-cases} is true.

\begin{prop}\label{prop:F-has-abs-irreducible-simple-factor}
If $F$ from Conjecture~\ref{conj-prime-degree-all-cases} is of the shape $F=F_1^aF_2^bG$ (or $F=F_1G$) where $F_1$ and $F_2$ are absolutely irreducible, $a$ and $b$ are coprime positive integers, and $F_1$ and $F_2$ do not divide $G$, then Conjecture~\ref{conj-prime-degree-all-cases} holds for such $F$.
\end{prop}

\begin{proof}
If $F=F_1G$, again, we conclude by Theorem~\ref{thm-Lewis-Schuur} as in the end of the proof of Theorem~\ref{thm-any-degree-dense-from-some-p}.  

If $F=F_1^aF_2^bG$, then by Theorem~\ref{thm-Lewis-Schuur}, $F_1$ and $F_2$ have a smooth root over $\F_p$ which is not a root of $F_2G$ and $F_1G$ respectively. Then, as in the end of the proof of Theorem~\ref{thm-any-degree-dense-from-some-p}, we would conclude that, when we fix all variables except from one, without loss of generality, we may assume that over $\F_p$, we have 
$F(x,c_1,\ldots,c_{m-1})=(x-z_1)^af_1(x)$ and $F(d_0,y,d_2,\ldots,d_{m-1})=(y-z_2)^bf_2(y)$, where $f_1(z_1)f_2(z_2)\neq 0$. Then, by Hensel's lemma for polynomials, we have the factorisation of the same type over $\Z_p$. Hence, $F$ has the property that after specifying all but one variable, it has a root of multiplicity $a$, and the same for $b$ (but it might require another variable). Then we conclude by the following lemma.

\begin{lem}\label{lem:as-for-polynomials-with-coprime-multiplicities}
Let $F\in \Z[x_0,\ldots,x_{m-1}]$ be a form. Assume there are $c_1,\ldots,c_{m-1}\in \Z_p$, and $d_0,d_2,\ldots,d_{m-1}\in \Z_p$ such that $F(x,c_1,\ldots,c_{m-1})=(x-z_1)^af_1(x)$ and $F(d_0,y,d_2,\ldots,d_{m-1})=(y-z_2)^bf_2(y)$, where $a$ and $b$ are coprime, then this lemma is the direct consequence of Theorem~\ref{thm:polynomial-coprime-multiplicities-roots}) over $\Z_p[x]$ and $\Z_p[y]$, where $f_1(z_1)f_2(z_2)\neq 0$. Then $R(F)$ is dense in $\Q_p$.   
\end{lem}
\begin{proof}
The proof follows the same as the proof of Theorem~\ref{thm:polynomial-coprime-multiplicities-roots}; here we even have more freedom than in Theorem~\ref{thm:polynomial-coprime-multiplicities-roots}.
\end{proof}
\end{proof}

We would like to use Lemma~\ref{lem:as-for-polynomials-with-coprime-multiplicities} to prove Conjecture~\ref{conj-prime-degree-all-cases}. The difficulty is to prove that we can apply it for any form $F$ from Conjecture~\ref{conj-prime-degree-all-cases}.

Unlike the counterexample in \cite[Remark 1.6]{miska}, where it is shown that for $f(x)=x^6(x+1)^{10}(x+2)^{15}\in \Z[x]$ of prime degree 31, we have that $R(f)$ is not dense in any $\Q_p$, here we hope that their strategy will not be applicable. Assume that $p>2$, then at most one of $x,x+1,x+2$ can be divisible by $p$, so it follows that $\nu_p(f(x))$ is always divisible by 6, 10 or 15, and hence, it is not possible to obtain $\nu_p(\frac{f(x)}{f(y)})=1$. 

In our situation, assume that $F=F_1^{\alpha_1}\cdots F_k^{\alpha_k}$. We have that $\sum_{i=1}^k \alpha_i\deg(F_i)$ is a prime number. What is different in this situation from \cite[Remark 1.6]{miska} is that we have many variables, so we can ask for a condition that $V_+(F_1,\ldots,F_{k-1})\setminus V_+(F_k)$ has a point over $\F_p$ (which does not work for polynomials in one variable). Then, we could hope that in this way, we get such a zero of $F_p$ over $\F_p$, which lifts to a zero of $F$, when we specify all variables except one, over $\Q_p$, which could be of a coprime multiplicity to a zero coming from a $\F_p$-point in $V_+(F_k)$ lifted to a zero of $F$, again after specifying all variables except from one. If we are able to do that, then we can conclude by Lemma~\ref{lem:as-for-polynomials-with-coprime-multiplicities}. We can consider any two different subsets of $\{F_1,\ldots,F_k\}$ and study their common zeros to prove that the form satisfies the conditions of Lemma~\ref{lem:as-for-polynomials-with-coprime-multiplicities}; we just mentioned one instance above.

However, checking that we can apply Lemma~\ref{lem:as-for-polynomials-with-coprime-multiplicities} seems extremely difficult, because we cannot control the multiplicities of zeros of $F_i$ (when specified to one variable), especially common zeros. Thus, we leave it as a conjecture, for which we explained why we believe it is true.
    
\begin{rem}\label{rmk:conjecture-holds-2,3,5}
Conjecture \ref{conj-prime-degree-all-cases} is true for small degrees: 2, 3, and 5. Namely, for degree 2, it follows from \cite{Donnay}, as explained in the introduction, and for degree 5, it is Theorem~\ref{thm-5}. Consider an integral form $F$ of degree 3 and $\o(F)\geq 4$. By Theorem~\ref{thm:cone-over-Fp-p-large-over-Q}, $\o(F_p)\geq 4$ for large enough $p$. Then by Chevalley-Warning theorem, see, for example, \cite[P. 5, Theorem 3]{Serre}, $F_p$ has a zero in $\F_p$. By \cite[Lemma 3.5]{leep}, this implies that $F_p$ has a non-singular zero in $\F_p$. Then, as before, we conclude that $R(F)$ is dense in $\Q_p$.   
\end{rem} 

\subsection{Denseness for one prime does not imply denseness for infinitely many of them}\label{subsec:density-one-p-not-infinitely}

Now we prove Theorem \ref{thm:dense-p-not-dense-almost-everywhere}, which follows immediately from the following more precise theorem. This theorem gives explicit examples of forms $F$ for which there are only finitely many, but at least one $p$, for which $R(F)$ is dense in $\Q_p$. The strategy is to construct a form for which $\nu_p(\frac{f(x)}{f(y)})$ can be any integer (and then to match the rest by Hensel's lemma), but that for almost all primes $q$, $\nu_q(\frac{f(x)}{f(y)})$ cannot be any integer, concretely, 1. The construction was inspired by and generalises \cite[Remark 1.6]{miska}.

\begin{thm}\label{thm:dense-p-not-dense-almost-everywhere-explicit}
Let $p$ be a prime number. Consider three prime numbers $q_1<q_2<q_3$ such that $q_1q_2q_3$ is coprime to $p(p-1)$. 
\begin{itemize}
\item[(a1)] Define the polynomial
\[f(x)=x^{q_1q_2}\prod_{i=1}^{q_1}(x+p^i)^{q_1q_3}(x+p^{q_1+1})^{q_2q_3}.\]
Then $R(f)$ is dense in $\Q_p$.
\item[(a2)] For almost all primes $q$, including those greater than $p^{q_1+1}$, $R(f)$ is not dense in $\Q_q$.
\item[(b1)] Define the polynomial
\[g_2(x_1,x_2)=x_1^{q_1q_2}\prod_{i=1}^{q_1}(x_1+p^ix_2)^{q_1q_3}(x_1+p^{q_1+1}x_2)^{q_2q_3}\]
\[\prod_{i=1}^{q_3-1}(p^ix_1+x_2)^{q_1q_2}\prod_{i=q_3}^{q_3+q_2-q_1-1}(p^ix_1+x_2)^{q_1q_3}\prod_{i=q_3+q_2-q_1}^{q_3+q_2-2}(p^ix_1+x_2)^{q_2q_3}.\]
Then $R(g_2)$ is dense in $\Q_p$.
\item[(b2)] For almost all primes $q$, including those greater than $p^{q_3+q_2+q_1-1}$, $R(g_2)$ is not dense in $\Q_q$.
\item[(c1)] Let $n\in \Z_{>2}$ and $g_n(x_1,\ldots,x_n)=g_2(x_1,x_2)(x_3\cdots x_n)^{q_1q_2q_3}$. Then, $R(g_n)$ is dense in $\Q_p$. 
\item[(c2)] For almost all primes $q$, including those greater than $p^{q_3+q_2+q_1-1}$, $R(g_n)$ is not dense in $\Q_q$.
\end{itemize}
\end{thm}

\begin{proof}
(a1) By Lemma \ref{lem:from-Zp-solutions-to-denseness}, it is sufficient to prove that for any $u\in \Z_p$, there are $x,y\in\Z_p$ such that $\frac{f(x)}{f(y)}=u$. Let $u=p^{nq_1+s}v$, where $n,s\in \Z_{\geq 0}$ with $s<q_1$, and $p\nmid v$. There is exactly one $m\in\{1,\ldots,q_1\}$ such that $q_2q_3(1-m)\equiv s\pmod{q_1}$. We will find $x,y\in \Z_p$ in the shape $x=p^kx_1$ and $y=p^ly_1-p^m$, where $k,l>q_1+1$ and $x_1,y_1\in\Z_p ^*$. For such $x$ and $y$, we have 
\[\nu_p(f(x))=kq_1q_2+q_1q_3\dfrac{q_1(q_1+1)}{2}+(q_1+1)q_2q_3;\]
\[\nu_p(f(y))=mq_1q_2+q_1q_3\left(\dfrac{(m-1)m}{2}+l+m(q_1-m)\right)+mq_2q_3;\]
\[
\dfrac{f(x)}{f(y)}=p^{\nu_p(f(x))-\nu_p(f(y))}\dfrac{x_1^{q_1q_2}(pf_1(x_1)+1)}{y_1^{q_1q_3}(pf_2(y_1)\pm1)},
\]
for some polynomials $f_1,f_2$ with coefficients in $\Z_p$, and for the unique choice of $+$ or $-$ in the denominator.

The choice of $m$ implies that there are $k,l>q_1+1$ such that $\nu_p(f(x))-\nu_p(f(y))=nq_1+s$ because we can divide by $q_1$ and obtain an equation $q_2k-q_3l=N$ for some $N\in\Z$. This equation has infinitely many solutions (and greater than $q_1+1$) because $q_2$ and $q_3$ are coprime. 
Now we need to prove that there are $x_1,y_1\in\Z_p ^*$ such that 
\begin{equation}\label{eq:after-division-by-p-proposition-dense-finitely-many-p}
\dfrac{x_1^{q_1q_2}(pf_1(x_1)+1)}{y_1^{q_1q_3}(pf_2(y_1)\pm1)}=v.    
\end{equation}

This directly follows from Hensel's lemma after multiplying,  fixing any $y_1\in \Z_p ^*$, and reducing the equation modulo $p$. We obtain an equation $x_1^{q_1q_2}\equiv \pm vy^{q_2q_3}\pmod{p}$, which has a simple solution because $x\mapsto x^{q_1q_2}$ is bijective map $\F_p ^ * \longrightarrow \F_p ^*$. We use that  $q_1q_2$ is coprime to $(p-1)p$.\\\\
(a2) Let $q$ be a prime such that $q\nmid p^i-p$, for any $1\leq i\leq q_1+1$. Then, for every $x\in \Z_q$, at most one of the following values is positive $\nu_q(x),\nu_q(x+p),\ldots,\nu_q(x+p^{q_1+1})$. Hence, for each $x\in \Z_q$, $\nu_q(f(x))$ is divisible by $q_1q_2$, or $q_2q_3$ or $q_1q_3$, so we cannot find $x,y\in \Z_q$ such that $\nu_q(f(x)/f(y))=1$, and $R(f)$ is not dense in $\Q_q$ by Lemma~\ref{lem:dense-has-all-valuations}. In particular, this holds for all $q$ greater than $p^{q_1+1}$.\\\\
(b1) We will prove that for any $u=p^{\nu_p(u)}v$, where $v\in \Z_p^*$, we can find $x_1, y_1\in \Z_p$ such that $\frac{g_2(x_1,1)}{g_2(y_1,1)}=u$. 

We note that $g_2(x_1,1)=f(x_1)(ph(x_1)+1)$, where $f$ is from (a1) and $h\in \Z_p[x]$ is some polynomial. Hence, $\nu_p(g_2(x_1,1))=\nu_p(f(x_1))$ for all $x_1\in \Z_p$. We can use the same strategy as in (a1) (i.e., look for $x_1,y_1\in \Z_p$ in the shape $x_1=p^kx_2$ and $y_1=p^ly_2-p^m$) to arrange that $\nu_p(\frac{g_2(x_1,1)}{g_2(y_1,1)})=\nu_p(u)$.
Then, after dividing by powers of $p$, we need to, using the notation of Equation \eqref{eq:after-division-by-p-proposition-dense-finitely-many-p}, solve the following equation
\[
\dfrac{x_2^{q_1q_2}(pf_1(x_2)+1)(ph(p^kx_2)+1)}{y_2^{q_1q_3}(pf_2(y_2)\pm1)(ph(p^ly_2-p^m)+1)}=v. 
\]
Now, as in (a1), by Hensel's lemma, we can find such $x_2$ and $y_2$ (even after fixing any $y_2\in\Z_p^*$), and hence $R(g_2)$ is dense in $\Q_p$.\\\\
(b2) First note that $g_2(x_1,x_2)$ has degree 
\[q_1q_2+q_1^2q_3+q_2q_3+(q_3-1)q_1q_2+(q_2-q_1)q_1q_3+(q_1-1)q_2q_3=3q_1q_2q_3.\]

We will now prove that for $q$ big enough, we cannot find $x_1,x_2,y_1,y_2\in\Z_p$ such that 
\[
\nu_q\left(\dfrac{g_2(x_1,x_2)}{g_2(y_1,y_2)}\right)=1.
\]
More precisely, we will prove that at least one of $q_1$, $q_2$, or $q_3$ divides $\nu_q\left(\frac{g_2(x_1,x_2)}{g_2(y_1,y_2)}\right)$, and for that, it suffices to prove that for any $x_1,x_2\in \Z_p$, two of $q_1$, $q_2$, and $q_3$ divide $\nu_q(g_2(x_1,x_2))$, so we now focus on proving the latter statement. 

We may assume that $q\nmid x_1$ or $q\nmid x_2$, because $g_2(qx_1,qx_2)=q^{3q_1q_2q_3}g_2(x_1,x_2)$. 

If $\nu_q(g_2(x_1,x_2))\neq 0$ (otherwise, the statement is true), then $q$ divides at least one of the terms $x_1$, $x_1+p^ix_2$ for $1\leq i\leq q_1+1$, or $p^jx_1+x_2$ for $1\leq j\leq q_2+q_3-2$. We now prove that $q$ cannot divide more than one of these terms. If $q$ divides at least two of them, then we distinguish between the cases. If $q\mid x_1$, then the second division implies that $q\mid p^ix_2$, for some $i$, hence $q\mid x_2$, which we assume it is not the case. If $q$ divides two terms of the same shape, then $q\mid (p^i-p^j)x_k$, where $k$ is 1 or 2, for $i,j<q_2+q_3-1$, which again implies that $q$ divides both $x_1$ and $x_2$. Similarly, if $q\mid x_1+p^ix_2$ and $q\mid p^jx_1+x_2$ for some $i\leq q_1+1$ and $j\leq q_2+q_3-2$, but then $q\mid (p^{i+j}-1)x_1$, which again implies that $q\mid x_1$ because $q>p^{q_1+q_2+q_3-1}$ and $q\mid x_2$. Thus, $R(g_2)$ is not dense in $\Q_q$.\\\\
(c1) Since $g_n(x_1,x_2,1,\ldots,1)=g_2(x_1,x_2)$, by (b1), $R(g_n)$ is dense in $\Q_p$. \\\\
(c2) Since $\nu_q(g_n(x_1,\ldots,x_n))\equiv \nu_q(g_2(x_1,x_2))\pmod{q_1q_2q_3}$ for any prime number $q$, the same argument as in (b2) shows that $R(g_n)$ is not dense for $q>p^{q_3+q_2+q_1-1}$.
\end{proof}

\subsection*{Acknowledgements} S. G.  was supported by the Spanish Ministry of Science and Innovation / State Research Agency (Ministerio de Ciencia e Innovación / Agencia Estatal de Investigación), through the Ramón y Cajal program, grant number RYC24SG-AD, MPIM guest postdoctoral fellowship program, Czech Science Foundation GA\v{C}R, grant 21-00420M, and Junior Fund grant for postdoctoral positions at Charles University during various stages of this project.  We would like to thank Du\v{s}an Dragutinovi\'{c}, Christopher Frei,
Gerard van der Geer, Christopher Keyes, Guido Lido, Steffen M\"{u}ller, Pieter Moree, Sun Woo Park, Lazar Radi\v{c}evi\'{c}, Efthymios Sofos, Jaap Top, and Pavlo Yatsyna for helpful discussions and suggestions that improved our article. Finally, 
we would like to thank the anonymous referee for the careful reading of the manuscript, numerous corrections, several useful suggestions, and a number of improvements and generalisations, especially for the whole Section~\ref{subsec:Schur-Dedekind}, from which this article benefited a lot.


\end{document}